\documentclass{amsart}

\usepackage{amssymb}
\usepackage{hyperref}
\usepackage[utf8x]{inputenc}

\DeclareMathOperator{\id}{id}
\DeclareMathOperator{\supp}{supp}
\DeclareMathOperator{\gr}{Gr}
\newtheorem{theorem}{Theorem}[section]
\newtheorem{lemma}[theorem]{Lemma}
\newtheorem{proposition}[theorem]{Proposition}

\numberwithin{equation}{section}

\title{Improved decay of conical averages of the Fourier transform}
\author[Terence L.~J.~Harris]{Terence L.~J.~Harris}
\address{Department of Mathematics, University of Illinois, Urbana, IL 61801, U.S.A.}
\email{terence2@illinois.edu}
\thanks{This material is based upon work partially supported by the National Science Foundation under Grant No. DMS-1501041.
The author would like to thank Burak Erdoğan for suggesting this problem and for advice on this topic, and for financial support.}

\begin{document}

\begin{abstract} An improved lower bound is given for the decay of conical averages of Fourier transforms of measures, for cones of dimension $d \geq 4$.  The proof uses a weighted version of the broad restriction inequality, a narrow decoupling inequality for the cone, and some techniques of Du and Zhang originally developed for the Schrödinger equation. 
\end{abstract}

\maketitle

\section{Introduction} 

The decay of the Fourier transform over submanifolds of Euclidean space is tied to various problems in geometric measure theory and partial differential equations. Averages over the sphere are connected to Falconer's distance set conjecture \cite{matilla,liu}, whilst the conical averages are equivalent to $L^2$ fractal Strichartz inequalities for the wave equation \cite{wolff,cho}, and have also been applied to Marstrand-type theorems for restricted families of projections \cite{oberlin3}, see \cite{iosevich} for a broad overview. The focus of this work is to improve the known decay rates for averages over the cone. 

For $d \geq 2$, let $\beta\big(\alpha, \Gamma^d \big)$ be the supremum over all $\beta \geq 0$ satisfying
\[ \int {\left|\widehat{\mu}(R\xi)\right|}^2 \, d\sigma_{\Gamma}(\xi) \lesssim_{\beta} \|\mu\| c_{\alpha}(\mu) R^{-\beta} \quad \text{for all $R>0$.} \]
Here $\|\mu\|$ is the total variation norm, $\sigma_{\Gamma}$ is the surface measure on the truncated cone 
\[ \Gamma^d := \{ (\xi,|\xi|) \in \mathbb{R}^{d+1}: 1 \leq |\xi| \leq 2 \},\]
and the Fourier transform of $\mu$ is 
\[ \widehat{\mu}(\xi) := \int e^{-2\pi i \langle x, \xi \rangle} \, d\mu(x). \]
The inequalities are assumed to be uniform over all $\alpha$-dimensional Borel measures $\mu$ supported in the unit ball of $\mathbb{R}^{d+1}$, where $\alpha$-dimensional means that 
\[ c_{\alpha}(\mu) := \sup_{ \substack{ x \in \mathbb{R}^{d+1}  \\ r >0} } \frac{ \mu (B(x,r) ) }{r^{\alpha}} < \infty. \]

Using some of the techniques from \cite{zhang}, the best known lower bound for $\beta\big(\alpha, \Gamma^d \big)$ will be improved here for $d \geq 4$. The exact value of $\beta\big(\alpha,\Gamma^2\big)$ was determined by Erdoğan \cite{erdogan2}. For $d \geq 3$, Cho, Ham and Lee \cite{cho} proved that
\begin{equation}  \label{cho} \beta\big(\alpha, \Gamma^d \big) \geq \begin{cases} \alpha &\text{ if } 0 < \alpha \leq \frac{d-1}{2} \\
  \frac{d}{4} + \frac{\alpha}{2} - \frac{1}{4}  &\text{ if } \frac{d-1}{2} < \alpha \leq \frac{d+3}{2}\\
\alpha-1 &\text{ if } \frac{d+3}{2} < \alpha \leq d+1, \end{cases} \end{equation}
which gives the exact value of $\beta\big(\alpha,\Gamma^3\big)$ in the case $d=3$. For large $\alpha$ and $d$, \eqref{cho} was further improved in \cite{me}. The main result of this work is the lower bound
\begin{equation} \label{betabound2} \beta\big(\alpha, \Gamma^d \big) \geq \alpha - 1 + \frac{d-\alpha}{d-1} \quad \text{for} \quad \frac{d+1}{2} < \alpha < d. \end{equation}
This is probably not sharp for $d \geq 4$, but it supersedes the lower bound from \cite{me} (the exact value $\beta\big(\alpha, \Gamma^d \big) = \alpha-1$ for $\alpha \in [d,d+1]$ is known in every dimension \cite{sjolin,cho}). For $\alpha \in \left( \frac{d+1}{2}, d\right)$, \eqref{betabound2} improves \eqref{cho} for $d \geq 4$ and ties \eqref{cho} if $d=3$.

For $\alpha \in \left( \frac{d}{2},d\right)$, the lower bound in \eqref{betabound2} is larger than $\frac{\alpha(d-1)}{d}$, which is the current best known lower bound for the spherical averages over the sphere of one dimension less, due to Du and Zhang \cite{zhang}. The Fourier analytic properties of the cone and the sphere/paraboloid of one dimension less are generally expected to be similar, see \cite{taoparkcity}. The proof of \eqref{betabound2} given here largely follows that of Du and Zhang for the sphere, but there are two significant differences. A weighted version of the $d$-broad restriction inequality of Ou and Wang \cite{ou} is used instead of the Bennett-Carbery-Tao $d$-linear restriction inequality (or equivalently the $d$-linear refined Strichartz inequality), which is better in the case of the cone. To make use of this requires a narrow decoupling inequality for the cone, the proof of which is slightly more involved than in the sphere/paraboloid case, since a plane through the origin may intersect the cone tangentially.

\subsection{Notation} 

Let $A: \mathbb{R}^{d+1} \to \mathbb{R}^{d+1}$ be the unitary defined through the standard basis by 
\[ e_i \mapsto e_i \text{ for $1 \leq i \leq d-1$,} \quad \frac{e_{d+1}+e_d}{\sqrt{2}} \mapsto e_d, \quad  \frac{e_{d+1}-e_d}{\sqrt{2}} \mapsto e_{d+1}. \]
The letter $E$ will usually denote the extension operator for the truncated cone $\Gamma$, given by
\[ Ef(x,t) = \int_{B(0,2) \setminus B(0,1)} e^{2\pi i \left(\langle \xi, x \rangle + |\xi|t \right) } f(\xi) \, d\xi. \]
The statement that $F$ is essentially supported in $U$ will mean that $|F(x)| \lesssim_N R^{-N}$ for $x \notin U$, for arbitrarily large $N$. The open $\epsilon$-neighbourhood of the set $U$ will be denoted by $\mathcal{N}_{\epsilon}(U)$. For a box (resp. ball) $B$, the set $CB$ will be the similar set with the same centre, but with side lengths (resp. radius) scaled by $C$.

\section{Narrow decoupling for the cone}
\label{narrowdecoupling}
To state the decoupling theorem for the truncated cone from \cite{demeter}, the set
\[ \Gamma = \Gamma^d = \{ (\xi,|\xi|) \in \mathbb{R}^{d+1}: 1 \leq |\xi| \leq 2 \},\]
is partitioned into \emph{caps at scale $K^{-1}$} of the form
\[ \tau = \{ (\xi,|\xi|) \in \mathbb{R}^{d+1} : \xi/|\xi| \in C, \quad 1 \leq |\xi| \leq 2 \}, \]
where $K$ is a large parameter, and the sets $C$ are spherical caps in $S^{d-1}$ of diameter $K^{-1}$, which partition the sphere. 

\begin{theorem}[{\cite[Theorem 1.2]{demeter}}] \label{decouplestandard} If $F = \sum_{\tau} F_{\tau}$ is a sum over disjoint caps in $\Gamma^d$ at scale $K^{-1}$, such that each $F_{\tau}$ has Fourier transform supported in the $K^{-2}$-neighbourhood of $\tau$, then for any $\epsilon >0$, 
\[ \|F\|_q \leq C_{\epsilon} K^{\epsilon} \left( \sum_{\tau} \|F_{\tau} \|_q^2 \right)^{1/2}, \quad q = \frac{2(d+1)}{d-1}. \]
\end{theorem}

The narrow decoupling inequality asserts that if $d \geq 3$ and the caps each have at least one point with unit normal in a $K^{-2}$-neighbourhood of a $(d-1)$-dimensional vector space, the exponent can be increased to $\frac{2(d-1)}{d-3}$, corresponding to two dimensions less. The proof uses the same idea as that of narrow decoupling for the paraboloid \cite[Proposition 5.5]{demeter} (see also \cite[Section 4]{guthlecture} for a more detailed proof), but there is an extra rescaling step needed to deal with the case where the vector space is almost tangent to the cone. This technicality requires the normals to be $K^{-2}$-close to the vector space, rather than $K^{-1}$ as in the case of the paraboloid. That $K^{-2}$ is sufficient is essentially a consequence of the following (straightforward) lemma.

\begin{lemma} \label{obvious} Let $2 \leq k \leq d$, $K \geq 1$, and suppose that $P \subseteq \mathbb{R}^d$ is a $(k-1)$-dimensional affine plane in $\mathbb{R}^d$, which has nonempty intersection with $S^{d-1}$. Then
\[ \mathcal{N}_{K^{-2}}(P) \cap S^{d-1} \subseteq \mathcal{N}_{C_dK^{-1}}\left(P \cap S^{d-1}\right), \]
where $C_d$ is a constant depending only on $d$. 
\end{lemma}
\begin{proof} The plane $P$ is contained in a $k$-dimensional plane through the origin. After applying a unitary it may therefore be assumed that 
\[ P = \{(x',0) \in \mathbb{R}^k \times \mathbb{R}^{d-k}: \langle x', c \rangle = \lambda \}, \]
for some unit vector $c \in \mathbb{R}^k$ and $\lambda \in \mathbb{R}$. Then $|\lambda | \leq 1$ since $P$ intersects $S^{d-1}$. Let $x \in \mathcal{N}_{K^{-2}}(P) \cap S^{d-1}$, and write $x = (x', x'')$. Then 
\[ |\langle x', c \rangle - \lambda | \leq K^{-2}, \quad |x''| \leq K^{-2}. \] 
Write $x' = (\lambda + O(K^{-2}))c + \mu c^{\perp}$, where $\mu \in \mathbb{R}$, $c^{\perp} \in \mathbb{R}^k$  is a unit vector orthogonal to $c$, and 
\[ \lambda^2 + \mu^2 = 1+O(K^{-2}). \]
Assume $\mu \geq 0$ without loss of generality. Let $y'= \lambda c + \sqrt{1-\lambda^2}c^{\perp}$, so that
\begin{align*} |x'-y'|^2 &= O(K^{-4}) + \left(\mu - \sqrt{1-\lambda^2}\right)^2  \\
&\leq O(K^{-4}) + C\left|\mu^2 -(1-\lambda^2)\right| = O(K^{-2}). \end{align*}
By taking square roots and considering the point $y=(y',0) \in P \cap S^{d-1}$, this proves the lemma.
\end{proof}

\begin{theorem} \label{narrower} Let $d \geq 3$. Suppose that $F = \sum_{\tau} F_{\tau}$ is a sum over caps in $\Gamma^d$ at scale $K^{-1}$, with each $\widehat{F_{\tau}}$ supported in a $K^{-2}$ neighbourhood of $\tau$, and suppose there is a $(d-1)$-dimensional vector space $V$, such that each cap has a point with normal in a $K^{-2}$ neighbourhood of $V$. Then for any $\epsilon >0$, 
\[ \|F\|_p \leq C_{\epsilon} K^{\epsilon} \left( \sum_{\tau} \|F_{\tau}\|_p^2 \right)^{1/2}, \quad p = \frac{2(d-1)}{d-3}. \] \end{theorem}

\begin{proof} Let $c, c' \in \mathbb{R}^{d+1}$ be orthogonal unit vectors such that 
\[ V = \{ x: \langle x, c \rangle = \langle x, c' \rangle =0\}. \]
After a rotation of the first $d$ coordinates, which fixes the cone, assume that 
\[ c= \left(0,\dotsc, 0, \lambda, \sqrt{1-\lambda^2}\right), \quad \text{and let} \quad c^{\perp} = \left(0, \dotsc ,0, -\sqrt{1-\lambda^2 }, \lambda \right),\]
where $|\lambda| \in [0,1]$, and in general $c^{\perp}$ is distinct from $c'$. Without loss of generality it may be assumed that $V$ has nonempty intersection with the backward half light cone
\[ \left\{ (\xi,-|\xi|) \in \mathbb{R}^{d+1} : \xi_d \geq 0 \right\} \setminus \{0\}, \] 
which implies that $\lambda \in \left[ \frac{1}{\sqrt{2}}, 1 \right]$. If $\lambda \leq \frac{1}{\sqrt{2}} + C_d K^{-2}$ then there are $\lesssim_d 1$ caps in the sum and the inequality is trivial, so it may be assumed that $\lambda \geq \frac{1}{\sqrt{2}} + C_d K^{-2}$, where $C_d$ is a large constant depending only on $d$, to be chosen later. 

 Let $U$ be the unitary defined through the standard basis by
\[ e_i \mapsto e_i \text{ for } 1 \leq i \leq d-1, \quad e_{d} \mapsto c^{\perp}, \quad e_{d+1} \mapsto c. \]

It will be shown that the projection of the support of $\widehat{F}$ onto some plane is close to a lower dimensional cone. To make this precise, let $x_0 \in \supp \widehat{F_{\tau}}$ and $(\xi, |\xi|) \in \tau$ be such that $|x_0-(\xi,|\xi|)| \leq K^{-2}$. By the normal assumption, applying Lemma \ref{obvious} with $k=d-1$ gives a point $\left(\widetilde{\xi}, \big| \widetilde{\xi} \big|\right)$ with $\big| \widetilde{\xi} \big| =|\xi|$ and $\big|\xi-\widetilde{\xi}\big| \lesssim K^{-1}$, such that the normal to the cone at $\left(\widetilde{\xi}, \big| \widetilde{\xi} \big|\right)$ lies in $V$. Hence
\[ \lambda \widetilde{\xi}_d  - \big| \widetilde{\xi} \big|\sqrt{1-\lambda^2 }  = 0. \]

Let $P$ be the projection onto $\{x:\langle x, c \rangle = 0\}$, and let $\widetilde{\eta} = U^* P \left(\widetilde{\xi}, \big| \widetilde{\xi} \big|\right)$. Then
\[ \widetilde{\eta} = \left(\widetilde{\xi}_1, \dotsc, \widetilde{\xi}_{d-1}, \lambda \big| \widetilde{\xi} \big| - \widetilde{\xi}_d\sqrt{1-\lambda^2 } , 0 \right), \quad \widetilde{\eta}_d =\big| \widetilde{\xi} \big| \left( \frac{2\lambda^2 -1}{\lambda} \right). \]
Write $\widetilde{\eta}' = (\widetilde{\eta}_1, \dotsc, \widetilde{\eta}_{d-1})$, so that
\[ \left| \widetilde{\eta}' \right|^2 = \big| \widetilde{\xi} \big|^2 - \widetilde{\xi}_d^2 = \big| \widetilde{\xi} \big|^2 \left( \frac{2 \lambda^2-1}{\lambda^2} \right),  \]
and 
\[ \widetilde{\eta} = \left(\widetilde{\eta}', \left| \widetilde{\eta}' \right| \sqrt{2\lambda^2 -1 }, 0 \right), \quad \frac{1}{3} \left(\frac{\sqrt{2\lambda^2-1}}{\lambda} \right) \leq \left| \widetilde{\eta}' \right| \leq 3 \left(\frac{ \sqrt{2\lambda^2 -1}}{\lambda} \right). \]
Therefore, define $T: \mathbb{R}^{d+1} \to \mathbb{R}^{d+1}$ by 
\[ x \mapsto \left(\frac{\lambda x_1}{\sqrt{2\lambda^2-1}}, \dotsc, \frac{\lambda x_{d-1}}{\sqrt{2\lambda^2 -1}}, \frac{\lambda x_d}{2\lambda^2-1}, x_{d+1} \right). \]
Let 
\[ \widetilde{z}= (\widetilde{\omega}, |\widetilde{\omega}|,0) = TU^* P\left(\widetilde{\xi}, \big| \widetilde{\xi} \big|\right), \quad z=(\omega,\omega_d,0) = TU^*P(\xi, |\xi|). \]
Then $\widetilde{z}$ lies in the lower dimensional truncated cone
\[ \Gamma' = \left\{ (y, |y|,0) \in \mathbb{R}^{d-1} \times \mathbb{R} \times \mathbb{R} : 1/3 \leq |y| \leq 3 \right\}. \]
Since $\big|\xi- \widetilde{\xi} \big|\lesssim K^{-1}$, the distance between $(\xi,|\xi|)$ and the tangent plane to the cone at $\left(\widetilde{\xi}, \big| \widetilde{\xi} \big|\right)$ is $\lesssim K^{-2}$. By the definition of $c$, the tangent plane $W$ at $\left(\widetilde{\xi}, \big|\widetilde{\xi}\big|\right)$ is parallel to $\{x: \langle x, y_0 \rangle = 0\}$ for some nonzero $y_0$ orthogonal to $c$. Hence the projection $P$ sends the tangent plane to $\Gamma$ at $\left(\widetilde{\xi}, \big| \widetilde{\xi} \big|\right)$ to a tangent plane to $P(\Gamma \cap G^{-1}(V))$, where $G$ is the Gauss map sending a point to its unit normal, and so the tangent plane to $\Gamma'$ at $\widetilde{z}$ is $(TU^* P) W$. Therefore, by the definition of $T$, the distance between $z$ and the tangent plane to $\Gamma'$ at $\widetilde{z}$ is $\lesssim \left(K  \sqrt{2 \lambda^2 -1} \right)^{-2}$.  Moreover, the condition $\big| \xi - \widetilde{\xi} \big|\lesssim K^{-1}$ gives
\[ |\omega -\widetilde{\omega}| \lesssim \left( K \sqrt{2\lambda^2 -1} \right)^{-1} \leq \left( K \sqrt{  C_d K^{-2} \sqrt{2}} \right)^{-1}. \]
By taking $C_d$ large enough, depending only on $d$, this gives $|\omega| \geq 1/3$. It follows that the distance from $z$ to the cone $\Gamma'$ is $\lesssim \left(K\sqrt{2\lambda^2 -1 } \right)^{-2}$. Hence each set $TU^*P \left(\supp \widehat{F_{\tau}} \right)$ is contained in a $\sim \left(K\sqrt{2\lambda^2 -1 }  \right)^{-2}$ neighbourhood of a cap $S(\tau)$ in $\Gamma'$ at scale $\sim \left(K\sqrt{2\lambda^2 -1} \right)^{-1}$.

The normal to the cone at $(\widetilde{\omega}, |\widetilde{\omega}|) = TU^*P\left(\widetilde{\xi}, \big|\widetilde{\xi}\big|\right)$ has direction $T^{-1}U^*\mathbf{n}$ where $\mathbf{n}$ is the unit normal to the cone at $\left(\widetilde{\xi}, \big| \widetilde{\xi} \big|\right)$. Hence for each cap $S(\tau) \subseteq \Gamma'$, the $\sim \left( K\sqrt{2\lambda^2 -1 } \right)^{-1}$ neighbourhood of $S(\tau)$ has a point in $\Gamma'$ whose normal lies in the $(d-1)$-dimensional vector space $T^{-1}U^*V$.    

Let 
\[ G_s(x) = (F \circ U \circ T)(x, s), \quad x \in \mathbb{R}^d, \quad s \in \mathbb{R}. \]
The Fourier transform of $F \circ U \circ T$ is $\left( \det T \right)^{-1}\widehat{F} \circ U \circ T^{-1}$, and so
\begin{equation} \label{GS} \widehat{G_s}(\xi_1, \dotsc, \xi_d) = \int e^{2\pi i s\xi_{d+1}} (\det T)^{-1} \left( \widehat{F} \circ U \circ T^{-1}\right)(\xi_1, \dotsc, \xi_d, \xi_{d+1}) \, d\xi_{d+1}, \end{equation}
which can be checked by taking the $d$-dimensional inverse Fourier transform of both sides. Let $\pi: \mathbb{R}^{d+1} \to \mathbb{R}^{d+1}$ be the projection $(x_1, \dotsc, x_d,x_{d+1}) \mapsto (x_1, \dotsc, x_d, 0)$. Then \eqref{GS} shows that the support of $\widehat{G_s}$ is contained in $\pi \left( \supp  \left(\widehat{F} \circ U \circ T^{-1}\right) \right)$. But since $\pi$ commutes with $T$ and $\pi U^* = U^* P$, this means that $\widehat{G_s}$ has support in the $\sim \left(K\sqrt{2\lambda^2 -1 } \right)^{-2}$ neighbourhood of the cone $\Gamma'$, and 
\begin{equation} \label{gsum} G_s = \sum_{\tau} G_{s, S(\tau)}, \end{equation}
is a sum over caps $S(\tau)$ in the cone $\Gamma'$ at scale $\left(K\sqrt{2\lambda^2 -1 }  \right)^{-1}$, such that the support of $\widehat{G_{s, S(\tau)}}$ is contained in the $\sim \left(K\sqrt{2\lambda^2 -1 }  \right)^{-2}$ neighbourhood of $S(\tau)$. By a change of variables,
\begin{equation} \label{steplower} \| F\|_p = \left( \det T \right)^{1/p} \left( \iint |G_s(x) |^p \, dx \, ds \right)^{1/p}.  \end{equation}
By Minkowski's inequality, to decouple $F$ it will suffice to decouple each $G_s$. But the only properties of $F$ used in obtaining \eqref{gsum} and \eqref{steplower} were that $F$ is a sum over $K^{-1}$-caps, and that there is a $d$-dimensional plane such that for each $\tau \in T$, there is a point in the $\sim K^{-1}$ neighbourhood of $\tau$ in the cone with normal lying in the $d$-dimensional plane. By the preceding working, these properties both apply to $G_s$ in one dimension less, with caps at scale $\left(K\sqrt{2\lambda^2 -1 } \right)^{-1}$, and so the same reasoning can be applied to each $G_s$ to get
\[  \int |G_s(x) |^p \, dx = \left(\det T' \right)\iint |H_{s,s'}(y) |^p \, dy \, ds', \]
where 
\[ H_{s,s'} = \sum_{\tau} H_{s,s', S'(\tau)}, \]
is a sum over caps $S'(\tau)$ at scale $\left( K\sqrt{ 2\lambda'^2 -1}\sqrt{2\lambda^2 -1 }  \right)^{-2}$ in the cone 
\[ \Gamma'' = \left\{ (y, |y|,0,0) \in \mathbb{R}^{d-2} \times \mathbb{R} \times \mathbb{R} \times \mathbb{R} : 1/4 \leq |y| \leq 4 \right\} \]
(the case where $\lambda' \leq \frac{1}{\sqrt{2}} + O\left(\left(K \sqrt{2\lambda^2-1}\right)^{-2} \right)$ can be dismissed as before, since there are $\lesssim_d 1$ caps in the sum). But now the standard decoupling theorem for the $(d-2)$-dimensional cone, Theorem \ref{decouplestandard}, can be applied to each $H_{s,s'}$ to get
\begin{align*}  \| F\|_p &= \left( \det T \right)^{1/p} \left( \det T' \right)^{1/p}\left( \iiint |H_{s,s'}(y) |^p \, dy \, ds' \, ds \right)^{1/p}  \\
&\leq C_{\epsilon} K^{\epsilon}\left( \det T \right)^{1/p} \left( \det T' \right)^{1/p}\left( \iint \left( \sum_{\tau} \left\|H_{s,s', S'(\tau)} \right\|_p^2 \right)^{p/2}  \, ds' \, ds \right)^{1/p}  \\
&\leq C_{\epsilon} K^{\epsilon} \left( \det T \right)^{1/p} \left( \det T' \right)^{1/p}\left( \sum_{\tau} \left(\iint  \left\|H_{s,s', S'(\tau)} \right\|_p^p \, ds' \, ds \right)^{2/p} \right)^{1/2} \\
&= C_{\epsilon} K^{\epsilon}\left( \sum_{\tau} \|F\|_p^2 \right)^{1/2}. \end{align*}
This finishes the proof. \end{proof}

\section{Fractal inequality via broad restriction} 

\label{s:multilinear}

The following wave packet decomposition is standard \cite{ou}; one derivation can be found in \cite{me}.
\begin{proposition} \label{wavepacket} Fix a small $\delta>0$, and let $K = R^{\delta}, R^{1/4}$ or $R^{1/2}$. Let $\tau$ be a cap in the cone at scale $K^{-1}$. Then any $f \in L^2(\mathbb{R}^d)$ supported in the projection $\pi(\tau) \cap B(0,2) \setminus B(0,1)$ of the cap $\tau$ onto $\mathbb{R}^d$ can be decomposed as $f= \sum_{\Box} f_{\Box}$, where each $f_{\Box}$ is supported in $(4/3) \pi(\tau)$ and 
\[ \sum_{\Box} \left\|f_{\Box} \right\|_2^2 \lesssim \|f\|_2^2. \]
The sets $\Box$ form a finitely overlapping cover of $\mathbb{R}^{d+1}$, each with dimensions
\[ \frac{RK^{\delta}}{K} \times \dotsm \times \frac{RK^{\delta}}{K} \times \frac{RK^{\delta}}{K^2} \times R, \]
with long axis normal to $\tau$ and short axis in the flat direction of $\tau$. The restriction of each $Ef_{\Box}$ to $B(0,R)$ is essentially supported in the set $\Box$, with
\[ \sum_{\Box: (x,t) \notin \Box} |Ef_{\Box}(x,t)| \lesssim_N R^{-N} \|f\|_2 \quad \text{if $|(x,t)| \leq R$}, \]
for arbitrarily large $N$. 
\end{proposition}

To prove the main fractal inequality of this section, a weighted version of the broad restriction inequality from \cite{ou} will be needed. The only novelty is the insertion of the weight into the proof from \cite{ou}, but for completeness most of the details will at least be sketched. The weight was used in a similar way in \cite[Eq. 5.10]{du}. 

Decompose the cone into caps $\tau$ at scale $K^{-1}$. For a point $x$ in the cone $\Gamma$, let $G(x)$ be the unit normal to the cone at $x$. Let $G(\tau)$ be the set of unit normals to points in $\tau$. For any vector space $V \subseteq \mathbb{R}^{d+1}$, define the angle between $G(\tau)$ and $V$ by
\[ \angle( G(\tau), V) = \min_{x \in \tau, v \in V} \angle( G(x), v ). \]
For an exponent $q$, an integer $k$ with $2 \leq k \leq d+1$, a large positive integer $A$, and a parameter $R>K^2$, define the \emph{broad norm} by
\[ \| Ef\|_{BL^q_{k,A}(B(0,R) \cap Y)}^q := \sum_{B_{K^2} \subseteq Y} \mu_{Ef}(B_{K^2}), \]
where the sum is over a union $Y$ of $K^2$-cubes $B_{K^2}$ in $B(0,R)$, and
\[\mu_{Ef}(B_{K^2}) :=  \min_{V_1, \dotsc, V_A \in \gr(k-1,d+1)} \max_{\substack{\tau \\
  \angle( G(\tau), V_a) \geq K^{-2} \, \forall a}} \|Ef_{\tau}\|_{L^q(B_{K^2})}^q, \]
for every $B_{K^2} \subseteq Y$. The set $\gr(k-1,d+1)$ is the set of $(k-1)$-dimensional subspaces of $\mathbb{R}^{d+1}$, and $\mu_{Ef}= \mu_{Ef,Y}$ is a measure extended by zero away from the cubes in $Y$.

\begin{lemma} \label{broadweight} Fix $d\geq 3$, $\epsilon,\delta,R, \alpha, \beta, \gamma >0$ and $K=R^{\delta}$. If $Y \subseteq \mathbb{R}^{d+1}$ is a union of $K^2$-cubes in $B(0,R)$ satisfying  
\begin{equation} \label{Ydim} \int_{B(x,r)} \chi_Y \, dy \leq \gamma r^{\alpha} \quad \text{for all $x \in \mathbb{R}^{d+1}$ and $r>K^2$,} \end{equation} 
and if 
\begin{equation} \label{alphabetacond} \frac{d+1}{2} < \alpha <d, \quad  \beta< \min\left\{ \alpha-1 + \frac{d-\alpha}{d-1}, \beta\big(\alpha, \Gamma^d \big) \right\} \end{equation}
then for $\delta = \delta(\epsilon) \ll \epsilon$ small enough, there is a constant $A= A(\epsilon)>0$ such that 
\[ \|Ef\|_{BL_{d,A}^q(B(0,R) \cap Y)} \leq C_{\epsilon} \gamma^{\lambda/2} R^{\epsilon} \|f\|_2 , \]
where 
\begin{equation} \label{qcond} \frac{1}{2} - \frac{1}{q} = \frac{1}{2d + \frac{1}{\alpha-\beta}}, \quad \lambda := \frac{\frac{1}{\alpha-\beta}}{2d+\frac{1}{\alpha-\beta}}. \end{equation}
\end{lemma}

\begin{proof}
By induction on $R$, assume the lemma holds for radii at most $R/2$. Let $D$ be a large constant to be chosen later. Using \cite[Theorem 5.5]{guth2}, there is a nonzero polynomial $P$ on $\mathbb{R}^{d+1}$ of degree $\lesssim D$, which is a product of $\sim \log D$ non-singular polynomials, whose zero set $Z(P)$ is such that $\mathbb{R}^{d+1} \setminus Z(P)$ has $\sim D^{d+1}$ connected components $O_i$, with the property that $\mu_{Ef}(O_i)$ is constant in $i$ up to a factor of 2. Let $W$ be the $R^{1/2 + \delta}$ neighbourhood of $Z(P)$. Then 
\begin{equation} \label{dichotomy2} \mu_{Ef}(B(0,R)) = \mu_{Ef}(W \cap B(0,R)) + \sum_i \mu_{Ef}(O_i \setminus W). \end{equation}
The sets $O_i \setminus W$ are called cells.

 If the cellular terms contribute at least $50\%$ to \eqref{dichotomy2}, then since the contribution of the sets $O_i$ are equal, at least $90\%$ of the $i$'s must satisfy 
\[ \mu_{Ef}(B(0,R)) \lesssim D^{d+1} \mu_{Ef}(O_i \setminus W). \]
Break $f$ up via the wave packet decomposition from Proposition \ref{wavepacket} with $K= R^{1/2}$, and let $f_i$ be the sum over the sets $\Box_{\tau}$ which intersect $O_i \setminus W$. Then
\[ \mu_{Ef}(O_i \setminus W) \lesssim \mu_{Ef_i}(O_i \setminus W) + R^{-N}\|f\|_2^q \leq \mu_{Ef_i}  (B(0,R)) + R^{-N}\|f\|_2^q, \]
for arbitrarily large $N$. The lemma has been assumed at scale $R/2$, and therefore holds up to a constant factor for $\mu_{Ef_i}(B(0,R))$. Hence 
\begin{equation} \label{bound1} \mu_{Ef_i}(O_i \setminus W) \lesssim \gamma^{(\lambda q)/2} R^{\epsilon q}\|f_i\|_2^q. \end{equation}
If a set $\Box_{\tau}$ intersects $O_i \setminus W$, then the centre line of $\Box_{\tau}$ intersects $O_i$. The restriction of $P$ to this centre line is a one-variable polynomial of degree at most $D$ which is not identically zero, and therefore has at most $D$ zeroes. Hence $D\|f\|_2^2 \gtrsim \sum_i \|f_i\|_2^2$, and therefore
\begin{equation} \label{bound2} \|f\|_2^2 \gtrsim D^{d} \|f_i\|_2^2, \end{equation}
for at least $90\%$ of the $i$'s. Hence \eqref{bound1} and \eqref{bound2} hold for at least $80\%$ of the $i$'s, and in particular the set $S$ of such $i$'s is nonempty. For $i \in S$,
\begin{align*} \|Ef\|_{BL_{d,A}^q(B(0,R) \cap Y)}^q &\lesssim D^{d+1} \mu_{Ef_i}(O_i \setminus W) + R^{-N}\|f\|_2^q \\
&\lesssim  D^{d+1}C \gamma^{(\lambda q) /2}R^{\epsilon q}\|f_i\|_2^q  + R^{-N} \|f\|_2^q\\
&\leq CD^{(d+1) - \frac{qd}{2}}  \gamma^{(\lambda q )/2}R^{\epsilon q} \|f\|_2^q. \end{align*}
The conditions $\alpha > \frac{d+1}{2}$ and $\beta < \alpha-1 + \frac{d-\alpha}{d-1}$ from \eqref{alphabetacond} combined with the definition of $q$ in \eqref{qcond} ensure that the exponent of $D$ is negative. At this point, choose the constant $D$ to be large enough to eliminate the implicit constants, so that the induction closes. This covers the case where the cellular terms dominate \eqref{dichotomy2}. 

Now suppose the non-cellular term dominates \eqref{dichotomy2}. By partitioning $Z(P)$ into $\sim \log D$ varieties, it may be assumed that the polynomial $P$ is nonsingular. Let $\{B_j\}_j$ be a covering of $B(0,R)$ by balls of a fixed radius $\rho < R/2$ (a sufficiently small constant multiple of $R$, to be chosen later). Define a set $\Box_{\tau}$ to be $R^{-1/2+\delta}$-tangent to $Z$ in $B_j$ if the following two conditions hold:
\begin{enumerate}
\item $\Box_{\tau} \cap 2B_j \subseteq N_{10R^{1/2 + \delta}}(Z) \cap 2B_j$;
\item $\angle \left(\Box_{\tau}, T_zZ \right) \leq R^{-1/2+\delta} \text{ for all } z \in Z \cap 2B_j \cap N_{100R^{1/2+\delta}} \Box_{\tau}$. \end{enumerate}
Let $\mathbb{T}_j := \left\{ (\tau, \Box_{\tau}): \Box_{\tau} \cap B_j \cap N_{R^{1/2+\delta}}(Z) \neq \emptyset  \right\}$, let
\[ \mathbb{T}_{j,tang} := \left\{ (\tau, \Box_{\tau}) \in \mathbb{T}_j: \text{$\Box_{\tau}$ is $R^{-1/2+\delta}$ tangent to $Z$ in $B_j$}  \right\}, \]
and $\mathbb{T}_{j,trans} := \mathbb{T}_j \setminus \mathbb{T}_{j,tang}$, and define $Ef_{j,tang}$, $Ef_{j,trans}$ accordingly. Since the non-cellular term dominates \eqref{dichotomy2}, 
\begin{align} \notag \mu_{Ef}(B(0,R)) &\lesssim \mu_{Ef}( W \cap B(0,R)) \\ 
\label{subcase} &\lesssim   \sum_j  \mu_{Ef_{j,trans}}( B_j) + \sum_j  \mu_{Ef_{j,tang}}(B_j) + R^{-N} \|f\|_2^q. \end{align} 

Suppose first that the transverse terms dominate \eqref{subcase}. Then since $q \geq 2$, 
\begin{align*} \mu_{Ef}(B(0,R)) &\lesssim \sum_j \mu_{Ef_{j, trans}}( B_j)  + R^{-N}\|f\|_2^q\\
&\lesssim \sum_j   \left[ \rho^{\epsilon}\right]^q \gamma^{(\lambda q)/2}\|f_{j, trans}\|_2^q  + R^{-N}\|f\|_2^q\\
&\lesssim \rho^{\epsilon q} \gamma^{(\lambda q)/2}\|f\|_2^q, \end{align*}
where to get from the second-last to the last line, transversality ensures there are $\lesssim 1$ overlaps in the sum, see \cite[Section 8.4]{guth2}. Since $\rho = R/C$, choosing the constant $C$ large enough closes the induction in the transverse subcase. 

For the remaining subcase, suppose the tangential terms dominate in \eqref{subcase}. Since $\beta < \beta\big(\alpha, \Gamma^d \big)$ by the assumption in \eqref{alphabetacond}, the inequality
\begin{equation} \label{wavepacket2} \|Ef_{j,tang}\|_{BL_{d,0}^2(B_j \cap Y)} \lesssim \left(\sum_{\tau} \|Ef_{j,tang,\tau}\|_{L^2(B_j \cap Y)}^2 \right)^{1/2} \lesssim \gamma^{1/2} \rho^{\frac{\alpha-\beta}{2}}  \|f_{j,tang}\|_2, \end{equation}
follows from a standard result of Wolff (see \cite[Proposition 5.3]{me} for a proof). By \cite[Theorem 4]{ou} with $n=d+1$, $k=m=d$ and $q_{d,d} := \frac{2d}{d-1}$,
\begin{equation} \label{ou-wang} \|Ef_{j,tang}\|_{BL_{d,A}^{q_{d,d}}(B_j\cap Y)}  \lesssim \|Ef_{j,tang}\|_{BL_{d,A}^{q_{d,d}}(B_j)} \lesssim K^{O(1)} \rho^{\epsilon} \rho^{\frac{-1}{4d}}\|f_{j,tang}\|_2. \end{equation}
The middle norm refers to the unweighted case where $Y=B_j$. The angle used here in the definition of the broad norm is less restrictive than in \cite{ou}, but (\cite{wang}) this does not harm the inequality. This is essentially since Lemma 2.2 and Equation 2.6 of \cite{ou} still hold. 

The definition of $\lambda$ in \eqref{qcond} satisfies
\[ \frac{1}{q} = \frac{\lambda}{2} + \frac{1-\lambda}{q_{d,d}}, \]
and so interpolation of \eqref{wavepacket2} and \eqref{ou-wang} via Hölder's inequality for the broad norm \cite[Lemma 4.2]{guth2} gives  
\[ \|Ef_{j,tang}\|_{BL_{d,A}^q(B_j \cap Y)} \lesssim \gamma^{\lambda/2} \rho^{\frac{\lambda(\alpha-\beta)}{2}} \rho^{\frac{-(1-\lambda)}{4d}} \rho^{(1-\lambda)\epsilon} K^{O(1)} \|f_{j,tang}\|_2. \]
The non-infinitesimal exponent of $\rho$ vanishes by \eqref{qcond}. Therefore summing over $j$ and using $q \geq 2$ gives
\[ \|Ef\|_{BL_{d,A}^q(B(0,R) \cap Y)}  \leq C_{\epsilon} \gamma^{\lambda/2} R^{\epsilon} \|f\|_2, \]
closing the induction in the tangential case. This finishes the proof. \end{proof}

The next lemma converts the preceding broad inequality to a linear one. It is formulated as an $L^2 \to L^p$ inequality with a parameter $K$ (which may essentially be thought of as equal to 1) in order to work well with $\ell^2$ decoupling for the $L^p$ norm and an induction on scale argument. The parameter $M$ will be eliminated when passing from an $L^2 \to L^p$ to an $L^2 \to L^2$ inequality. Although some steps in the proof are similar to those in \cite{zhang}, and also \cite{ou} and \cite{guth2}, most details will be included for completeness.

\begin{lemma} \label{prop2} Let $d \geq 3$, $\epsilon, \delta,\alpha,\beta, R, \gamma>0$, $K= R^{\delta}$, $p= \frac{2(d-1)}{d-3}$ and
\[ \frac{d+1}{2} < \alpha < d, \quad \beta < \min\left\{ \alpha-1 + \frac{d-\alpha}{d-1} , \beta\big(\alpha, \Gamma^d \big)\right\}. \]
Let $Y = \bigcup_{k=1}^M B_k$ be a union of disjoint $K^2$-cubes in $B(0,R)$, which are all translates of each another. Suppose that $\|Ef\|_{L^p(B_k)}$ is constant in $k$ up to a factor of 2, and that
\begin{equation} \label{fractal} \int_{B(x,r)} \chi_Y \leq \gamma r^{\alpha} \quad \text{for all $x \in \mathbb{R}^{d+1}$ and $r>K^2$. } \end{equation}
Let $k(\alpha,\beta):= \min\left\{ \frac{2(d-\alpha)}{\alpha-1}, \frac{1}{\alpha-\beta} \right\}$. Then for $\delta=\delta(\epsilon) \ll \epsilon$ small enough,
\[ \left\| Ef \right\|_{L^p(Y) } \leq C_{\epsilon} M^{\frac{-1}{d-1}} \gamma^{ \frac{1}{d-1}} R^{ \frac{\alpha}{2d+k(\alpha,\beta)} + \epsilon} \|f\|_2. \] \end{lemma}

\begin{proof}  Partition the truncated cone into caps at scale $K^{-1}$. Let $T$ be the set of caps. Fix the constants $\delta = \delta(\epsilon)$ and $A=A(\epsilon)$ from Lemma \ref{broadweight}.  For each cube $B \subseteq Y$, define the set of ``significant'' caps by 
\[ \mathcal{S}(B) = \left\{ \tau \in T: \|Ef_{\tau}\|_{L^p(B)} \geq \frac{1}{2|T|} \left\| Ef \right\|_{L^p(B)} \right\}. \] Fix some $K^2$-cube $B \subseteq Y$. Choose a collection of $A$ $(d-1)$-dimensional subspaces $V_a$, depending on $B$, attaining the minimum 
\[  \max_{\substack{\tau \\
\angle (G(\tau), V_a) \geq K^{-2} \, \forall a}} \int_B |Ef_{\tau}|^p = \min_{W_1, \dotsc,  W_a \in \gr(d-1, d+1)} \max_{\substack{\tau \\
\angle (G(\tau), W_a) \geq K^{-2} \, \forall a} }\int_B |Ef_{\tau}|^p. \]
Henceforth the notation $\tau \in V$ will be used to indicate $\angle (G(\tau), V) < K^{-2}$. If there exists $\tau \in \mathcal{S}(B)$ such that $\angle (G(\tau), V_a) \geq K^{-2}$ for all $a$, then 
\[  \int_B |Ef|^p \lesssim K^{O(1)} \max_{ \tau \notin V_a \, \forall a} \int_B |Ef_{\tau}|^p. \]
Otherwise $\int_B |Ef|^p \lesssim \sum_{a=1}^A \int_B \left| \sum_{\tau \in V_a \setminus \bigcup_{b=1}^{a-1} V_b } Ef_{\tau} \right|^p$ by the triangle inequality. In either case
\begin{align} \label{dichotomy}  \int_B |Ef|^p &\lesssim K^{O(1)} \min_{W_1, \dotsc , W_a \in \gr(d-1, d+1)} \max_{\tau \notin W_a \, \forall a}\int_B |Ef_{\tau}|^p \\
\notag &\quad +  \sum_{a=1}^A \int_B \left| \sum_{\tau \in V_a \setminus \bigcup_{b=1}^{a-1} V_b} Ef_{\tau} \right|^p. \end{align}
The cube $B$ is called broad if the first term dominates, and narrow if the second term dominates. Since $\|Ef\|_{L^p(B_k)}$ is essentially constant in $k$, it suffices to bound $\|Ef\|_{L^p(Y)}$ in two distinct cases; at least half of the cubes in $Y$ are broad, or at least half of the cubes in $Y$ are narrow.

In the broad case, it may be assumed that all of the cubes are broad. Define $q$ by \eqref{qcond}. By the uncertainty principle, 
\begin{equation} \label{uncertainty} \int_B |Ef_{\tau}|^p \lesssim K^{O(1)} \|Ef_{\tau}\|_{L^q(2B)}^p + R^{-N} \|f\|_2^p, \end{equation}
for every cap $\tau$, and for arbitrarily large $N$. 

%To see this, let $\phi$ be a non-negative Schwartz function equal to 1 on the ball $B(0, 2\sqrt{2})$ containing $\Gamma$ and vanishing outside the larger ball $B(0,3)$. Then $\widehat{Ef_{\tau}} = \widehat{ Ef_{\tau}} \phi$ gives $|Ef_{\tau}| \lesssim |Ef_{\tau}| \ast \zeta$, where $\zeta(x) = \frac{1}{1+|x|^N}$ for some fixed arbitrarily large $N$. Since $\zeta$ is essentially constant at scale one, so is $|Ef_{\tau}| \ast \zeta$. Hence the $L^p$-norm of $|Ef_{\tau}| \ast \zeta$ over any unit cube is comparable to the $L^q$-norm. Picking out the unit cube with maximal contribution to the integral and then invoking Minkowski's inequality to remove the $\zeta$ gives \eqref{uncertainty}. 

 By pigeonholing, there is a subset $Y' \subseteq Y$ consisting of a fraction $\geq \frac{1}{K^{O(1)}}$ of the cubes, such that
\[ \min_{W_1, \dotsc , W_a \in \gr(d-1, d+1)} \max_{\tau \notin W_a \, \forall a}\int_{2B} |Ef_{\tau}|^q, \]
is essentially constant as $B$ ranges over $Y'$. Summing \eqref{dichotomy} over $B \subseteq Y'$ therefore yields
\begin{align*} \int_Y |Ef|^p  &\lesssim K^{O(1)} M^{1- \frac{p}{q}} \left( \sum_{B \subseteq Y'} \min_{W_1, \dotsc , W_a \in \gr(d-1, d+1)} \max_{\tau \notin W_a \, \forall a}\int_{2B} |Ef_{\tau}|^q \right)^{p/q}  \\
&\quad + R^{-N} \|f\|_2^p \\
 &\lesssim K^{O(1)} M^{1- \frac{p}{q}} \| Ef\|_{BL^q_{d,A}(B(0,2R) \cap 2Y)}^p + R^{-N} \|f\|_2^p. \end{align*} 
Applying Lemma \ref{broadweight} and the definition of $k(\alpha,\beta)$ gives
\begin{align*} \|Ef\|_{L^p(Y)} &\leq C_{\epsilon}K^{O(1)}  M^{\frac{1}{p} - \frac{1}{q}} \gamma^{\frac{\lambda}{2}}R^{\frac{\epsilon}{2}}\|f\|_2 \\
&\leq C_{\epsilon}K^{O(1)} M^{\frac{1}{p} - \frac{1}{2}} \gamma^{\frac{1}{2}-\frac{1}{q}+\frac{\lambda}{2}} R^{\alpha \left( \frac{1}{2}- \frac{1}{q} \right) + \frac{\epsilon}{2}}\|f\|_2 \\
&\leq C_{\epsilon} M^{\frac{-1}{d-1}} \gamma^{\frac{1}{d-1}} R^{\frac{\alpha}{2d+k(\alpha,\beta)} + \epsilon} \|f\|_2. \end{align*}
 The change in the exponent of $\gamma$ is permissible since $\frac{1}{2}-\frac{1}{q}+\frac{\lambda}{2} \leq \frac{1}{d-1}$, and since $\gamma \gtrsim \frac{1}{K^{O(1)}}$, which follows from \eqref{fractal} with $r=2K^2$, and $x$ the centre of some cube in $Y$. This proves the theorem in the broad case.

In the narrow case, it may be assumed that all of the cubes are narrow. Using the wave packet decomposition from Proposition \ref{wavepacket} with $K= R^{\delta}$, decompose each $f_{\tau}$ as $f_{\tau} = \sum_{\Box_{\tau}} f_{\Box_{\tau}}$, where the sets $\Box_{\tau}$ form a finitely overlapping cover of physical space, and have dimensions 
\[ \frac{RK^{\delta}}{4K\sqrt{d+1}} \times \dotsm \times \frac{RK^{\delta}}{4K\sqrt{d+1}} \times \frac{RK^{\delta}}{4K^2\sqrt{d+1}} \times \frac{R}{4\sqrt{d+1}}, \]
with short axis in the flat direction in $\tau$, and long axis normal to $\tau$. Correspondingly $f= \sum_{\Box} f_{\Box}$ where each set $\Box$ corresponds to some $\tau$, but the cap is suppressed in the notation. Let $\widetilde{R} = \frac{RK^{\delta}}{K^2}$, let $\widetilde{K} = \widetilde{R}^{\delta}$ and make the inductive assumption that the theorem holds at scale $\widetilde{R}$. For each $\tau$, partition physical space into sets $S$ of dimensions
\[ \frac{\widetilde{K}^2 K}{2} \times \dotsm \times \frac{\widetilde{K}^2 K}{2} \times \frac{\widetilde{K}^2}{2} \times \frac{\widetilde{K}^2 K^2}{2}, \]
again with short axis in the flat direction of $\tau$, and long axis normal to $\tau$. 

For each $\tau$ let $\{\eta_S\}_S$ be a smooth partition of unity with each $\eta_S$ non-negative, $\eta_S \sim 1$ on $S$, essentially supported on $2S \cap N_{K^{2+\delta}}(S)$, with $\widehat{\eta_S}$ supported in a box around the origin of dimensions 
\begin{equation} \label{decouplingdim} K^{-2} \times \dotsm \times K^{-2}\times \widetilde{K}^{-2+\delta} \times K^{-2}, \end{equation} with long axis corresponding to the flat direction in $\tau$. This partition can be obtained by applying the Poisson summation formula at scale one, rescaling by the dimensions in \eqref{decouplingdim}, and then grouping the functions together with scaled lattice points in $S$. 

 For a given set $\Box$, sort the boxes $S$ with $2S \cap \Box \neq \emptyset$ into sets $\mathbb{S}_{\kappa}$ according to the dyadic value $\kappa$ of $\|Ef_{\Box}\|_{L^p(2S)}$. Partition these further into sets $\mathbb{S}_{\kappa,\eta}$, where $\eta$ is a dyadic number corresponding to the number of cubes $B \subseteq Y$ such that $K^{2\delta}B \cap S \neq \emptyset$. Let $Y_{\Box,\kappa,\eta}$ be the union of sets $S$ inside $\mathbb{S}_{\kappa, \eta}$, let
\[ Bx = (Kx_1, \dotsc , Kx_{d-1}, x_d, K^2x_{d+1}), \]
and define
\begin{equation} \label{gammadef} \widetilde{\gamma}=  \widetilde{\gamma}\left(\kappa, \eta, \Box\right) = \sup_{\substack{x \in \mathbb{R}^{d+1} \\
r> \widetilde{K}^2}}  \frac{1}{r^{\alpha}}\int_{B(x,r)} \chi_{U^*A^*B^{-1}A2Y_{\Box,\kappa, \eta}}, \end{equation}
where $U$ is a rotation which fixes the cone and carries $\tau$ to the cap with centre line in the direction $\frac{e_{d+1}+e_d}{\sqrt{2}}$. Define
\[ \eta_{Y_{\Box,\kappa, \eta}} = \sum_{S \subseteq Y_{\Box, \kappa, \eta}} \eta_S, \quad \text{so that} \quad Ef = \sum_{ \kappa,\eta} \sum_{\Box } \eta_{Y_{\Box,\kappa, \eta}} Ef_{\Box}. \]
Since $B$ is narrow, \eqref{dichotomy} becomes
\[ \|Ef\|_{L^p(B)}^p \lesssim \sum_{a=1}^A \int \Bigg| \sum_{\kappa, \eta} \sum_{\substack{\Box_{\tau} \\
 \tau \in V_a \setminus \bigcup_{b=1}^{a-1} V_b}} \eta_{Y_{\Box,\kappa, \eta}} Ef_{\Box} \Bigg|^p. \]

Hence for each cube $B \subseteq Y$, by the triangle inequality, there is a triple $(\eta,\kappa, a)$ independent of $\Box$ but dependent on $B$, such that
\begin{equation} \label{pigeontriangle} \|Ef\|_{L^p(B)} \lesssim K^{O(\delta)} \Bigg\| \sum_{\substack{\Box_{\tau} \\
 \tau \in V_a \setminus \bigcup_{b=1}^{a-1} V_b}} \eta_{Y_{\Box,\kappa,\eta}} Ef_{\Box} \Bigg\|_{L^p(B)} + R^{-N} \|f\|_2. \end{equation}
By pigeonholing, there is a fixed triple $(\eta,\kappa,a)$ independent of $\Box$, such that \eqref{pigeontriangle} holds for a fraction $\gtrsim \frac{1}{K^{O(\delta)}}$ of the cubes in $Y$. Therefore let $Y_{\Box} = Y_{\Box,\kappa, \eta}$ for this choice of $\eta$. By pigeonholing the remaining cubes again, there is a subset $\mathbb{B}$ of sets $\Box$ such that $\|f_{\Box}\|_2$ is essentially constant over $\Box \in \mathbb{B}$, each $Y_{\Box}$ corresponding to $\Box \in \mathbb{B}$ contains $\sim \widetilde{M}$ sets $S$ in $Y_{\Box}$, and the inequality  
\begin{equation} \label{pigeontriangle2} \|Ef\|_{L^p(B)} \lesssim K^{O(\delta)} \Bigg\| \sum_{\substack{\Box_{\tau} \in \mathbb{B} \\
\tau \in V_a \setminus \bigcup_{b=1}^{a-1} V_b }} \eta_{Y_{\Box}} Ef_{\Box} \Bigg\|_{L^p(B)} + R^{-N} \|f\|_2, \end{equation}
holds for a fraction $\gtrsim \frac{1}{K^{O(\delta)}}$ of the cubes $B \subseteq Y$. By further pigeonholing the remaining cubes, there is a dyadic number $\mu$ and a set $Y'$ consisting of a fraction $\gtrsim \frac{1}{K^{O(\delta)}}$ of the cubes in $Y$, such that for each $B \subseteq Y'$ the cube $K^{2\delta}B$ intersects $\sim \mu$ different sets $Y_{\Box}$ as $\Box$ ranges over $\mathbb{B}$, and \eqref{pigeontriangle2} holds for all $B \subseteq Y'$. 

The sum in \eqref{pigeontriangle2} satisfies the conditions for narrow decoupling; for each fixed $\tau$ the sets $\Box_{\tau}$ have (much) larger side lengths than each $B \subseteq Y'$, which means that for each $\tau$ and $B$ there are at most $\sim 1$ sets $\Box_{\tau}$ such that $K^{2\delta}B \cap Y_{\Box} \neq \emptyset$. Moreover, the Fourier transform of $\eta_{Y_{\Box}}Ef_{\Box}$ is supported in a $\sim K^{-2}$ neighbourhood of $2\tau$, since the long direction in the support of $\widehat{\eta_{Y_{\Box}}}$ is in the flat direction of $\tau$. For each $B \subseteq Y'$, Theorem \ref{narrower} therefore gives, 
\begin{align*} \|Ef\|_{L^p(B)} &\lesssim K^{O(\delta)} \Bigg( \sum_{\substack{\Box_{\tau} \in \mathbb{B} \\
\tau \in V_a \setminus \bigcup_{b=1}^{a-1} V_b \\
K^{2\delta}B \cap Y_{\Box} \neq \emptyset}} \left\|\eta_{Y_{\Box}} Ef_{\Box} \right\|_{L^p(2B)}^2 \Bigg)^{1/2} + R^{-N} \|f\|_2 \\
&\lesssim K^{O(\delta)} \mu^{ \frac{1}{2}- \frac{1}{p}} \left( \sum_{\Box \in \mathbb{B}} \left\|\eta_{Y_{\Box}} Ef_{\Box} \right\|_{L^p(2B)}^p \right)^{1/p} + R^{-N} \|f\|_2. \end{align*}
Since the cubes are disjoint and contribute equally, summing over $B \subseteq Y'$ gives

\begin{equation} \label{summation} \|Ef\|_{L^p(Y)} \lesssim K^{O(\delta)} \mu^{\frac{1}{d-1} } \left( \sum_{\Box \in \mathbb{B}} \|Ef_{\Box}\|_{L^p(2Y_{\Box})}^p \right)^{1/p} + R^{-N} \|f\|_2. \end{equation}

To apply Lorentz rescaling to a given summand, assume after a rotation that the cap $\tau$ corresponding to the set $\Box$ has centre line in the direction $\frac{e_{d+1}+e_d}{\sqrt{2}}$ (so that $U=\id$ in \eqref{gammadef}). Use the change of variables
\[ (\eta, |\eta| ) = A^*BA (\xi,|\xi|), \quad g(\eta) = K^{\frac{d-1}{2}}  \left| \frac{d\xi}{d\eta} \right| f_{\Box} (\xi), \]
so that $\|g\|_2 \sim \|f_{\Box}\|_2$ and 
\begin{align} \notag \|Ef_{\Box}\|_{L^p(2Y_{\Box})} &= K^{\frac{d+1}{p}- \frac{(d-1)}{2}}\|Eg\|_{L^p(Z)} \\
\label{tbd}  &\leq C_{\epsilon} K^{\frac{-2}{d-1}}\widetilde{M}^{\frac{-1}{d-1}} \widetilde{\gamma}^{\frac{1}{d-1} } \widetilde{R}^{\frac{\alpha}{2d+k(\alpha,\beta)}+\epsilon} \|f_{\Box}\|_2. \end{align}
To verify this inequality, the set $Z$ is defined by $Z= A^* B^{-1} A 2Y_{\Box}$. It is a union of $\sim \widetilde{M}$ cubes $A^*B^{-1}A2S$ of side length $\widetilde{K}^2$ which are all translates of each other, all contained inside a ball of radius $\widetilde{R}$, and all of which contribute equally to the integral. By selecting out a fraction $\sim_d 1$ of the cubes, it may be assumed that the cubes are disjoint. The definition of $\widetilde{\gamma}$ in \eqref{gammadef} is 
\[ \widetilde{\gamma} = \sup_{\substack{x \in \mathbb{R}^{d+1} \\
r> \widetilde{K}^2}}  \frac{1}{r^{\alpha}}\int_{B(x,r)} \chi_Z. \] Therefore, using the inductive assumption and applying the theorem at scale $\widetilde{R}$ gives the inequality \eqref{tbd}.

By definition of $\mu$,
\begin{align}\notag  M\mu &\lesssim K^{O(\delta)} \sum_{B \subseteq Y'} \sum_{\substack{\Box \in \mathbb{B} \\
 K^{2\delta}B \cap Y_{\Box} \neq \emptyset }} 1 \\
\notag &\leq K^{O(\delta)} \sum_{\Box \in \mathbb{B}} \sum_{S \subseteq Y_{\Box}} \sum_{\substack{B \subseteq Y' \\ K^{2\delta}B \cap S \neq \emptyset}} 1 \\
\label{mubound} &\lesssim K^{O(\delta)} | \mathbb{B} | \widetilde{M} \eta. \end{align}

To bound $\widetilde{\gamma}= \widetilde{\gamma}( \Box, \kappa, \eta)$, fix some $\Box \in \mathbb{B}$, assume without loss of generality that $U= \id$, let $x \in \mathbb{R}^{d+1}$ and $r> \widetilde{K}^2$ be given. Then
\begin{align*} \frac{1}{r^{\alpha}}\int_{B(x,r)} \chi_{A^*B^{-1}A2Y_{\Box}} \, dz &= \frac{1}{r^{\alpha}K^{d+1}} \int_{(A^*BA)B(x,r)} \chi_{2Y_{\Box}} \, dy \\
&\leq \frac{1}{r^{\alpha}K^{d+1}} \sum_{S \subseteq Y_{\Box}} \int_{(A^*BA)B(x,r)} \chi_{2S} \, dy \\
&\lesssim \frac{1}{r^{\alpha}K^{d+1}} \sum_{\substack{S \subseteq Y_{\Box} \\
S \subseteq (A^*BA)B\left(x,4r \sqrt{d+1}\right)}} \widetilde{K}^{2(d+1)}K^{d+1} \\
&\lesssim \frac{K^{O(\delta)}}{\eta r^{\alpha}} \sum_{\substack{B \subseteq Y \\
B \subseteq (A^*BA)B\left(x,K^{10\delta}r\right)}} \widetilde{K}^{2(d+1)} \\
&\leq \frac{K^{O(\delta)}}{\eta r^{\alpha}} \int_{(A^*BA)B(x,K^{10\delta}r)} \chi_Y \, dy \\
&\lesssim \frac{K^{1+ \alpha+ O(\delta)}\gamma}{\eta},  \end{align*}
where the last line follows from covering $(A^*BA)B(x,K^{10\delta}r)$ by $\lesssim K$ balls of radius $K^{1+O(\delta)}r$ and applying \eqref{fractal}. Taking the supremum over $r > \widetilde{K}^2$ gives
\begin{equation} \label{gammabound} \widetilde{\gamma} = \widetilde{\gamma}(\Box,\kappa, \eta) \lesssim K^{1+\alpha+O(\delta)} \gamma \eta^{-1}. \end{equation}

Putting \eqref{tbd},\eqref{mubound} and \eqref{gammabound} into \eqref{summation} yields
\begin{align*} \|Ef\|_{L^p(Y)} &\leq C_{\epsilon}C_{\epsilon}' K^{\frac{\alpha+1}{d-1} -\frac{2}{d-1} - \frac{2\alpha}{2d+k(\alpha,\beta)} -2\epsilon + O(\delta)}M^{\frac{-1}{d-1}} \gamma^{ \frac{1}{d-1}} R^{ \frac{\alpha}{2d+k(\alpha,\beta)}+\epsilon}\|f\|_2 \\
&\leq C_{\epsilon}M^{\frac{-1}{d-1}} \gamma^{ \frac{1}{d-1}} R^{ \frac{\alpha}{2d+k(\alpha,\beta)}+\epsilon}\|f\|_2, \end{align*}
for $R$ large enough, by the definition of $k(\alpha,\beta)$. Therefore the induction closes in the narrow case, and this finishes the proof. \end{proof}

By pigeonholing and Hölder's inequality, Lemma \ref{prop2} implies the following $L^2 \to L^2$ inequality. 

\begin{lemma} \label{omitted} Let $d \geq 3$, $\alpha, \beta, R, \gamma>0$ with
\[ \frac{d+1}{2} < \alpha < d, \quad \beta < \min\left\{ \alpha-1 + \frac{d-\alpha}{d-1} , \beta\big(\alpha, \Gamma^d \big)\right\}. \]
Let $X$ be a union of $\mathbb{Z}^{d+1}$-lattice unit cubes in $B(0,R)$ with
\[ \int_{B(x,r)} \chi_X \leq \gamma r^{\alpha} \quad \text{for all $x \in \mathbb{R}^{d+1}$ and $r>1$. } \]
Then for any $\epsilon >0$,
\begin{equation} \label{fractal2} \left\| Ef \right\|_{L^2(X) } \leq C_{\epsilon} \gamma^{\frac{1}{d-1} }R^{\frac{ \alpha}{2d+k(\alpha,\beta)}+\epsilon} \|f\|_2. \end{equation}
\end{lemma}
Successively iterating the preceding three lemmas results in the following theorem. 
\begin{theorem} For $d \geq 3$ and $\frac{d+1}{2} < \alpha < d$,
\begin{equation} \label{lowerbound} \beta\big(\alpha, \Gamma^d \big) \geq  \alpha-1 + \frac{d-\alpha}{d-1}. \end{equation}
\end{theorem}
\begin{proof} For a contradiction, assume that \eqref{lowerbound} does not hold. Successively applying Lemmas \ref{broadweight}, \ref{prop2} and \ref{omitted} with $\beta<\beta\big(\alpha, \Gamma^d \big)$, and then letting $\beta \to \beta\big(\alpha, \Gamma^d \big)$ gives \eqref{fractal2} with $\beta= \beta\big(\alpha, \Gamma^d \big)$. By duality, pigeonholing and the definition of $\beta\big(\alpha, \Gamma^d \big)$ (see for example \cite{du}),
\[ \beta\big(\alpha, \Gamma^d \big) \geq \alpha \left( 1- \frac{2}{2d+k(\alpha,\beta\big(\alpha, \Gamma^d \big))} \right) = \alpha \left( 1- \frac{2}{2d+\frac{1}{\alpha-\beta\big(\alpha, \Gamma^d \big)}} \right), \]
where the equality comes from the assumption that \eqref{lowerbound} fails. The iterated function $f: (0,\alpha) \to (0, \alpha)$ defined by $f(x) = \alpha \left( 1- \frac{2}{2d+ \frac{1}{\alpha-x}} \right)$ is increasing on its domain and has a globally attracting fixed point at $x= \alpha-\frac{\alpha}{d}+\frac{1}{2d}$. Hence
\[ \beta\big(\alpha, \Gamma^d \big) \geq \alpha-\frac{\alpha}{d} + \frac{1}{2d} > \alpha-1 + \frac{d-\alpha}{d-1}, \]
since $\alpha > \frac{d+1}{2}$. This is a contradiction. \end{proof}

\end{document}